\documentclass[a4paper,11pt,reqno]{amsart}

\usepackage[margin=1 in]{geometry}
\usepackage[singlespacing]{setspace}
\usepackage{fancyhdr}
\usepackage[inline]{enumitem}
\usepackage{newclude}
\usepackage{caption}
\usepackage{subcaption}
\usepackage[utf8]{inputenc}
\usepackage[T1]{fontenc}
\usepackage{lmodern}
\usepackage{appendix}
\usepackage{comment}
\usepackage{graphicx}

\usepackage{amssymb}
\usepackage{dsfont}

\usepackage{amsthm}
\usepackage{mathtools}
\usepackage{bm}
\usepackage{mathrsfs}
\usepackage{bbm}
\usepackage{xcolor}

\usepackage{microtype}

\usepackage[colorlinks=true,citecolor=black,linkcolor=black,urlcolor=black]{hyperref} 

\newcommand{\bP}{\mathbf{P}}
\newcommand{\bE}{\mathbf{E}}

\newcommand{\mc}{\mathcal}

\newcommand{\gD}{\Delta}

\newcommand{\veps}{\varepsilon}

\newcommand{\HitH}{{T_{H}}}
\newcommand{\HitG}{{T_{G}}}

\newcommand{\HitHP}{{T'_{H}}}

\newcommand{\Bone}{{\mathcal{B}_1}}
\newcommand{\Btwo}{{\mathcal{B}_2}}
\newcommand{\Bthree}{{\mathcal{B}_3}}
\newcommand{\Bfour}{{\mathcal{B}_4}}
\newcommand{\Bfive}{{\mathcal{B}_5}}

\DeclareMathOperator{\aut}{aut}
\DeclareMathOperator{\Bin}{Bin}

\DeclareMathOperator{\cl}{c_F}

\theoremstyle{definition}
\newtheorem{definition}{Definition}[section]
\newtheorem{remark}[definition]{Remark}

\newtheorem*{calgorithm}{Coupling Algorithm}
\theoremstyle{plain}
\newtheorem{theorem}[definition]{Theorem}
\newtheorem{proposition}[definition]{Proposition}
\newtheorem{lemma}[definition]{Lemma}
\newtheorem{corollary}[definition]{Corollary}

\title{The hitting time of nice factors}
\author[F. Burghart, M. Kaufmann, N. Müller, M. Pasch]{Fabian Burghart, Marc Kaufmann, Noela Müller, Matija Pasch}
\date{September 26, 2024}

\thanks{All authors wish to thank Annika Heckel for the significant impact she had on the present article. F. B. has received funding from the European Union's Horizon 2020 research and innovation programme under the Marie Sk\l{}odowska-Curie Grant Agreement no. 101034253. M. K. gratefully acknowledges support by the Swiss National Science Foundation [grant number 200021\_192079]. N. M. is supported by the NWO Gravitation project NETWORKS under grant no. 024.002.003.  }

\address{Fabian Burghart, {\tt f.burghart@tue.nl}, Eindhoven University of Technology, Department of Mathematics and Computer Science, MetaForum MF 4.091, 5600 MB Eindhoven, the Netherlands.}

\address{Marc Kaufmann, {\tt marc.kaufmann@inf.ethz.ch}, ETH Zürich, Department of Computer Science, Andreasstrasse 5, 8050 Zürich, Switzerland.}

\address{Noela M\"uller, {\tt n.s.muller@tue.nl}, Eindhoven University of Technology, Department of Mathematics and Computer Science, MetaForum MF 4.084, 5600 MB Eindhoven, the Netherlands.}

\address{Matija Pasch, {\tt matija.pasch@gmail.com}, Munich, Germany.}

\begin{document}

\begin{abstract}
 Consider the random $u$-uniform hypergraph (or $u$-graph) process on $n$ vertices, where $n$ is divisible by $r>u\ge 2$. It was recently shown that with high probability, as soon as every vertex is covered by a copy of the complete $u$-graph $K_r$, it also contains a $K_r$-factor (RSA, Vol. 65 II, Sept. 2024). 
 The hitting time result is obtained using a process coupling, which is based on the proof of the corresponding sharp threshold result (RSA, Vol. 61 IV, Dec. 2022). The latter, however, was not only derived for complete $u$-graphs, but for a broader class of so-called nice $u$-graphs.

 The purpose of this article is to extend the process coupling for complete $u$-graphs to the full scope of the sharp threshold result: nice $u$-graphs. As a byproduct, we obtain the extension of the hitting time result to nice $u$-graphs. Since the relevant combinatorial bounds in the proof for the $K_r$-case cannot be generalized, we introduce new arguments that do not only apply to nice u-graphs, but will be relevant for the broader class of strictly 1-balanced u-graphs. Further, we show how the remainder of the process coupling for the $K_r$-case can be utilized in a black-box manner for any u-graph. These advances pave the way for future generalizations.
\end{abstract}

\maketitle

\section{Introduction}

Which properties of a graph guarantee the existence of a perfect matching? A partial answer to this question was given by Petersen over 130 years ago~\cite{petersen1891}.
Sixty years later, in 1947, it was completely resolved by Tutte~\cite{tutte1947factorization}.

In the setting of hypergraphs and introducing randomness, the following closely related question arises: 
Under which conditions are we \emph{likely} to find a perfect matching in the \emph{random $r$-uniform hypergraph} $H_r(n,\pi)$ with vertex set $[n]=\{1,\dots,n\}$, where each hyperedge is present independently with probability $\pi$? For the graph case $r=2$, this was answered by Erd\H{o}s and Rényi \cite{erdos1966existence}, who observed that $p^*=p^*(n)=\ln(n)/n$ is a \emph{sharp threshold} for the existence of a perfect matching -- in the following sense. Let $\varepsilon>0$ be arbitrarily small.
Then for any sequence $p\le (1-\veps)p^*$, the \emph{random graph} $G(n,p)=H_2(n,p)$ does not contain a perfect matching with high probability (whp), that is, with probability tending to one as $n$ increases.
On the other hand, for $p\ge (1+\veps)p^*$, the random graph $G(n,p)$ does contain a perfect matching whp.

The surprisingly challenging extension of this result to hypergraphs, that is, to $r>2$, became known as Shamir's problem. Since hypergraphs with isolated vertices cannot contain perfect matchings, and since the sharp threshold for the disappearance of isolated vertices and the sharp threshold for the existence of perfect matchings coincide in the graph case $r=2$, it was conjectured early on that the sharp threshold 
\[
 \pi_0=\binom{n-1}{r-1}^{-1}\log(n)
\]
for the disappearance of isolated vertices in $H_r(n,\pi)$ is also the sharp threshold for the existence of perfect matchings for any $r\ge 2$. Yet, it took 40 years until Kahn established that this is indeed the case in his breakthrough papers~\cite{kahn2022hitting,kahn2019asymptotics}:

\begin{theorem}[{\cite[Theorem 1.4]{kahn2019asymptotics}}]\label{thm:kahn_threshold}
 For any $r \geq 3$ and $\varepsilon >0$, whp the $r$-uniform random hypergraph $H_r(n, (1+\varepsilon)\pi_0)$ contains a perfect matching.
\end{theorem}

Remarkably, shortly after the appearance of \cite{kahn2019asymptotics}, Riordan and Heckel \cite{riordan2022random, heckel2020random} were able to employ Theorem \ref{thm:kahn_threshold} to obtain the sharp threshold for the existence of $K_r$-factors in $G(n,p)$. Recall that a $K_r$-factor is a spanning subgraph composed of vertex-disjoint copies of the complete graph $K_r$. Clearly, these questions are related: If we replace each hyperedge of a given $r$-uniform hypergraph with a perfect matching by a copy of $K_r$, we obtain a graph on the same vertex set that contains a $K_r$-factor. We can also apply our previous observation regarding isolated vertices and perfect matchings to the new setup: Say a vertex is $K_r$-isolated if it is not contained in a copy of $K_r$, then a graph with $K_r$-isolated vertices cannot contain a $K_r$-factor.

The sharp threshold $p_0$ for the existence of $K_r$-isolated vertices in $G(n,p)$ is for example established in \cite[Theorem~3.22]{janson2000random}. Again, guided by existing results for perfect matchings, that is the case $r=2$, it was conjectured that $p_0$ is also the sharp threshold for the existence of a $K_r$-factor in $G(n,p)$. Now, instead of the laborious task of adapting the proof of Theorem \ref{thm:kahn_threshold} to the new setup, Riordan and Heckel devised an ingenious coupling.

\begin{theorem}[\cite{riordan2022random,heckel2020random}]\label{thm:coupling_Kr}
 For any $r>2$ there exist $\delta,\veps>0$ such that the following holds.
 Let $p\le n^{-2/r+\veps}$ and $\pi=(1-n^{-\delta})p^{\binom r2}$, then there exists a coupling of $G(n,p)$ and $H_r(n,\pi)$ such that whp for every hyperedge $h$ in $H_r(n,\pi)$ the copy of $K_r$ on the vertex set $h$ is contained in $G(n,p)$.
\end{theorem}

From this coupling it is not only immediate that $p_0\le n^{-2/r+\veps}$ is the sharp threshold for the existence of a $K_r$-factor in $G(n,p)$, much more than that, it gives a rigorous foundation for the correspondence of the hypergraph model and the graph model that we described above.

\subsection{Sharper than sharp thresholds}
Already back in 1985, Bollob\'as and Thomason \cite{bollobas1985}
observed that the close relationship between the disappearance of isolated vertices and the emergence of perfect matchings goes much deeper than the mere equality of the thresholds. Let $N_r=\binom{n}{r}$ and let $(H_t^r)_{t=0}^{N_r}$ be the standard \emph{random hypergraph process}, where we start with the empty hypergraph $H_0$ with vertices $[n]$ and add one of the hyperedges of size $r$, which are not already present, uniformly at random in each time step.
Let
\[
 T_H^r = \mathrm{min}\{ t : H_t^r \text{ has no isolated vertices}\}
\]
be the time $t$ where the last isolated vertex disappears.
For the case $r=2$, that is, for the graph process $(G_t)_t=(H^2_t)_t$, Bollob\'as and Thomason showed that, whp, as soon as we reach $T_H^2$, that is, the earliest possible time, the graph $G_{T_H^2}$ contains a perfect matching.
Kahn \cite{kahn2022hitting} extended this beautiful result to $r$-uniform hypergraphs for all $r \ge 3$.

\begin{theorem}[{\cite[Theorem 1.3]{kahn2022hitting}}]\label{theorem:kahnhitting}
 Let $r \ge 3$ and $n \in r \mathbb Z_+$, then whp $H^r_{T^r_H}$ has a perfect matching.
\end{theorem}

The catch with the coupling in Theorem \ref{thm:coupling_Kr} is that it is one-sided: We `embed' $H_r(n,\pi)$ into $G(n,p)$. Thus, at first, it seemed unlikely that an analogous coupling approach could be used to translate Theorem \ref{theorem:kahnhitting} to $K_r$-factors. Nevertheless, Heckel, Kaufmann, Müller and Pasch \cite{heckel2024hitting} recently provided an extension of the `static' coupling to a process coupling, thereby enabling the transferral of the hitting time result. 

\begin{theorem}[{\cite[Theorem 1.6]{heckel2024hitting}}] \label{thm:coupling_Proc}
 For any $r\ge 3$, there exists a coupling of $G_{T_G}$ and $H^r_{T^r_H}$ such that whp for every hyperedge $h$ of $H^r_{T^r_H}$ the copy of $K_r$ with vertices $h$ is contained in $G_{T_G}$.
\end{theorem}

As before, this result  does not only directly imply that whp, a $K_r$-factor exists in the hitting time graph, it further showcases the connection between the two models.

\subsection{Nice hypergraphs}\label{sec:nice_hypergraphs}
The story so far seems neatly wrapped up, with all plotlines resolved. But this is far from being true: There is a long-standing conjecture \cite[below Conjecture 1.1]{johansson2008factors} that Theorem~\ref{thm:coupling_Proc}  does in fact not only hold for $K_r$, but for any strictly 1-balanced $u$-uniform hypergraph. For brevity and to stay close to \cite{riordan2022random}, we henceforth refer to $u$-uniform hypergraphs as $u$-graphs.

When we consider more general $u$-graphs $F$ than complete $u$-graphs, more than one copy of $F$ might be present on a given vertex set. Hence, simple $r$-uniform hypergraphs cannot adequately encode copies of $F$ anymore. 
Instead, we resort to \emph{$F$-graphs}, which are $r$-uniform labeled multi-hypergraphs with vertices $[n]$, where hyperedges are labelled by the possible copies of $F$ on the hyperedge vertex set. 
Note that for $F=K_r$, an $F$-graph is just a simple $r$-uniform hypergraph (if we ignore the labels).
Further, we may define the \emph{random $F$-graph} $H_F(n,\pi)$ with vertices $[n]$, where we include each $F$-edge independently with probability $\pi$.

While the sharp threshold result for strictly $1$-balanced $u$-graphs is still missing, Riordan's coupling result Theorem \ref{thm:coupling_Kr} covers much more than just the case $F=K_r$. 
Let a $u$-graph be \emph{nice} if it is strictly 1-balanced and $3$-connected, plus a technical condition for $u=2$, which is made explicit in Section \ref{sec:graphs_hypergraphs} below.
Examples for the graph case $u=2$ are complete graphs $K_r$ for $r\ge 4$ and complete bipartite graphs $K_{m,n}$ for $m,n\ge 3$.

\begin{theorem}[{\cite[Theorem 18]{riordan2022random}}]\label{theorem:riordancoupling}
 Let $F$ be a fixed nice $u$-graph, $u \geq 2$, with $v(F) = r \geq 4$ and $e(F)=s$. Let $d_1=s/(r-1)$. Then there are positive constants $\varepsilon(F),\delta(F)>0$ such that, for any $p=p(n) \le n^{-1/d_1 + \varepsilon}$, letting $\pi=(1-n^{-\delta})p^{s}$, we may couple the random $u$-graph $G=H_u(n,p)$ with the random $F$-graph $H=H_F(n,\pi)$ so that, whp, for every $F$-edge in $H$ the corresponding copy of $F$ in $G$ is present.
\end{theorem}
\begin{remark} \label{rem:constants}
 While the constant $\varepsilon(F)$ is explicit in \cite{riordan2022random}, Theorem 18 in \cite{riordan2022random} only mentions an unspecified $o(1)$-term in place of $n^{-\delta}$. However, the existence of the constant $\delta(F)$ can be extracted from the proof in \cite{riordan2022random}, as mentioned in \cite[Remark 4]{riordan2022random} for the case $F=K_r$.
\end{remark}

Given the coupling, the sharp threshold for the existence of $F$-factors follows as a corollary.

\begin{theorem}[{\cite[Theorem 10]{riordan2022random}}]\label{thm:riordan_sharp}
 For a fixed nice $u$-graph $F$ with $v(F)=r$ and $e(F)=s$,
 \[
  p_0 = ((\mathrm{aut}(F)/r)n^{-r+1} \log n)^{1/s}
 \]
 is a sharp threshold for $H_u(n,p)$ to contain an $F$-factor.
\end{theorem}

\subsection{Main result}
The purpose of this article is to establish the hitting time version of Theorem~\ref{theorem:riordancoupling}, to showcase the adaptability of the coupling approach in \cite{heckel2024hitting}, and to work towards general strictly $1$-balanced $u$-graphs.
Thus, for the random $u$-graph process $(G_t)_t=(H^u_t)_{t=0}^{N_u}$, let
\[
 T_G = \mathrm{min}\{ t : \text{ every vertex in $G_t$ is contained in at least one copy of $F$}\}
\]
be the time where the last vertex is covered by a copy of $F$. We also need a process version of the random $F$-graph $H_F(n,\pi)$ in Section \ref{sec:nice_hypergraphs}. Let $\aut(F)$ be the number of automorphisms of $F$, then the total number of $F$-edges on the vertex set $[n]$ is
\[
 M=\binom{n}{r}\frac{r!}{\aut(F)}.
\]
Now, the \emph{random $F$-graph process} $(H_t)_{t=0}^M$ starts with the empty $F$-graph $H_0$ with vertices $[n]$, and in each step we add a new $F$-edge, uniformly at random.
Let
\[
 T_H = \mathrm{min}\{ t : H_t\text{ has no isolated vertices}\}
\]
be the time where the last isolated vertex in $H_t$ disappears. Then the following holds.

\begin{theorem} \label{theorem:finalcoupling}
 Let $F$ be a fixed nice $u$-graph with $r \ge 4$ vertices. Then we may couple the random $u$-graph process $(G_t)_{t=0}^{N_u}$ with the random $F$-graph process  $(H_t)_{t=0}^{M}$ so that, whp, for every hyperedge in $H_\HitH$ the corresponding copy of $F$ is contained in $G_{\HitG}$.
\end{theorem}

Theorem~\ref{theorem:kahnhitting} and Theorem~\ref{theorem:finalcoupling} together imply that vertices not contained in any copy of $F$ are essentially the only obstruction for an $F$-factor in the random graph process.

\begin{corollary} \label{thm:Ffactor_HT}
 Let $F$ be a fixed nice $u$-graph with $r \ge 4$ vertices. If $r \ge 3$ and $n \in r \mathbb Z_+$, then whp $G_\HitG$ contains an $F$-factor.
\end{corollary}

\subsection{Proof strategy}
While we do build on the ideas from \cite{heckel2024hitting}, we derive a novel bound which will be crucial in the proof for general strictly 1-balanced $u$-graphs. Further, we present a short argument which applies to any $u$-graph $F$ and allows to use the last two parts of the proof in  \cite{heckel2024hitting} in a `black-box' manner. In fact, we obtain all process-related parts in \cite{heckel2024hitting} for general $u$-graphs. 
The details and the general approach are described below.

In Section~\ref{Sec:Stat}, analogously to Section~3 in \cite{heckel2024hitting}, we establish properties of Riordan's ``static'' coupling.
The key result here is that, whp, copies of $F$ in $H_u(n,p)$ that are not present as $F$-edges in $H_F(n,\pi)$ cannot be incident to vertices of ``low'' degree.
For the proof, we devise a new combinatorial argument in Lemma~\ref{lemma:different_estimate} that is a viable alternative to the technical optimization problems in the proofs of Lemma 3.8 and Lemma A.2 in \cite{heckel2024hitting} -- which do not extend beyond the complete $u$-graphs considered therein.
More than that, the ideas behind the new argument extend to all strictly 1-balanced $u$-graphs.

The actual process coupling in \cite{heckel2024hitting} was established in three parts. In Section \ref{Sec:Proc}, we adapt the proof from Section 4 from \cite{heckel2024hitting} to establish the first part of the process coupling. This is unavoidable due to the extension of our considerations from complete graphs $K_r$ to more general $u$-graphs $F$. But since this part of our proof does not rely on any specific property of $F$, it applies to \emph{all} $u$-graphs $F$, provided the availability of a static coupling that embeds $H_F(n,\pi)$ into $H_u(n,p)$ with the properties of Lemmata~\ref{lemma:cpl_bad}-\ref{lemma:lowdegnoextra}.
Then, we use a short argument which allows us to transfer the second and third part from \cite{heckel2024hitting} in a black-box manner. This argument does not only save us from the tedious task of adapting the proofs of these parts, it also again holds for general $u$-graphs. 
 
Finally, in the case of $K_r$-factors, it was necessary to distinguish the cases $r\geq 4$ and $r=3$. Our proofs here generally follow the ``easier'' case $r\geq 4$. The case of triangles is more complicated due to the presence of so-called clean 3-cycles, see \cite{heckel2020random}. Indeed, the class of nice $u$-graphs has been tailored in \cite{riordan2022random} to avoid such complications for the static coupling.
Our results show that the treatment of nice u-graphs is also readily available for the process version of the coupling.

\section{Preliminaries}\label{section:preliminaries}

\subsection{Graphs and hypergraphs}\label{sec:graphs_hypergraphs}
For convenience, we identify hypergraphs with their hyperedge set and $F$-graphs with their $F$-edge set. Given a (labeled multi-)hypergraph $H$, we write $v(H)$, $e(H)$ and $c(H)$ for the number of vertices, hyperedges and connected components of $H$, respectively. 

For $u \geq 2$, a \emph{$u$-graph} $G$ is a $u$-uniform simple hypergraph on the vertex set $[n]$. 
Next, we formally define \emph{nice} $u$-graphs, starting with their properties.
\begin{definition}[{1-balanced, \cite{riordan2022random}}]
 The 1-density of a $u$-graph $F$ with $v(F)>1$ is
 \[
  d_1(F) = \frac{e(F)}{v(F)-1}.
 \]
 Further, $F$ is \emph{1-balanced} if $d_1(F') \leq d_1(F)$ for all $F' \subsetneqq F$ with $v(F')>1$, and \emph{strictly 1-balanced} if this inequality is strict for all $F'\subsetneqq F$ with $v(F')>1$.
\end{definition}
In a certain sense, being $1$-balanced means that the graph $F$ does not contain a subgraph that is denser than itself.
Further details can be found in \cite{janson2000random}.

\begin{definition}[{$k$-connected}]
 A hypergraph $F$ is \emph{$k$-connected} if it has at least $k+1$ vertices and no vertex set $S\subset V(F)$ of size at most $k-1$ whose removal disconnects $F$.
\end{definition}

We are now in the position to define \textit{nice} $u$-graphs.

\begin{definition}[{Nice $u$-graphs, \cite{riordan2022random}}]\label{def:nice}
 A $u$-graph $F$ is \textit{nice} if \emph{(i)} $F$ is strictly $1$-balanced, \emph{(ii)} $F$ is $3$-connected and \emph{(iii)} either $u \geq 3$, or $u=2$ and $F$ cannot be transformed into an isomorphic graph by adding one edge and deleting one edge.
\end{definition}

From here on, fix a nice $u$-graph $F$ on $r$ vertices, note that $r\ge 4$, and let $s=e(F)$.
For the convenience of the reader, we repeat the definition of $F$-graphs.

\begin{definition}[$F$-graph, \cite{riordan2022random}]\label{def:F-graph}
 An \textit{$F$-graph} $H_F$ is a labeled multi-hypergraph $(V,E)$, where $V$ is a finite set, the labels are copies of $F$ with vertices in $V$ and the hyperedges (or $F$-edges) $h_F\in E$ are the vertex sets of their labels $F$.
\end{definition}

Let $\aut(F)$ be the number of automorphisms of $F$. Recall that the total numbers $M$ of $F$-edges and $N$ of $u$-edges are given by
\[
 M=\binom{n}{r}\frac{r!}{\mathrm{aut}(F)} \qquad \text{ and } \qquad N=\binom{n}{u}.
\]
For a $u$-graph $G$ with vertex set $V$, we can encode the copies of $F$ in $G$ 
by an $F$-graph $H$, by using exactly the copies of $F$ in $G$ as labels for $H$. We denote this $F$-graph by $\cl(G)$. 

Conversely, we will work with several objects derived from $F$-graphs, which will have a prominent role in the remainder.
Recall that the $F$-edge labels of $H$ are copies of $F$.
\begin{definition}[From $F$-graphs to (multi-)hypergraphs]\label{def:f-to-hyper}
 For an $F$-graph $H_F$, let
 \begin{itemize}
  \item $\tilde{H}_F$ be the multi-hypergraph obtained from $H_F$ by replacing each $F$-edge by a hyperedge with the same vertex set (i.e. forgetting the labels),
  \item $\hat{H}_F$ be the simple hypergraph obtained from $\tilde H_F$ by removing multiple hyperedges,
  \item $G(H_F)$ be the $u$-graph union of the hyperedge labels of $H_F$.
 \end{itemize}
\end{definition}

Finally, we recall the random $F$-graph. 
\begin{definition}[Random $F$-graph]
 For $n \geq 1$ and $\pi \in [0,1]$, the random $F$-graph $H_F(n, \pi)$ has vertex set $[n]$ and each of the $M$ possible $F$-edges (`copies of $F$') is present independently with probability $\pi$. 
\end{definition}

\subsection{Critical window for labeled hypergraphs}\label{section:critwin}
In this subsection, we describe the critical  window in which the last isolated vertex joins an $F$-edge in the random $F$-graph $H_F=H_F(n,\pi)$. 
First, we determine the distribution of the simple hypergraph $\hat H_F$ from Definition~\ref{def:f-to-hyper}.  

\begin{lemma}[\cite{riordan2022random}]\label{lem:pi'}
 Let $H_F = H_F(n, \pi)$. Then $\hat H_F$ is distributed like $H_r(n, \pi')$, where
 \[
  \pi' = 1 - (1-\pi)^{r!/\mathrm{aut}(F)}.
 \]
\end{lemma}
\begin{proof}
 Each hyperedge in $\hat{H}_F$ is present independently with probability $1 - (1-\pi)^{r!/\mathrm{aut}(F)}$.
\end{proof}

For the remainder of this note, we fix an arbitrary function $g(n)$ satisfying
\begin{equation} \label{eq:defg}
 g(n) = o\big(\log n/\log \log n \big)\quad  \text{ and } \quad g(n) \rightarrow \infty.
\end{equation}
The probabilities corresponding to the start and end of the critical window in $H_F(n,\pi)$ are
\begin{align}
 \pi_{\pm} = \frac{\mathrm{aut}(F)}{r!} \cdot \frac{\log n \pm g(n)}{\binom{n-1}{r-1}}.
\end{align}

Recall from \cite[Lemma~5.1a]{devlin2017} that, whp, $H_r(n, \pi')$ has isolated vertices for $\pi' \leq \pi_-$ and that, whp, $H_r(n,\pi')$ has no isolated vertices for $\pi' \geq\pi_+$.
It follows from Lemma~\ref{lem:pi'} that, whp, $H_F(n,\pi)$ has isolated vertices for $\pi \leq \pi_-$, and that, whp, $H_F(n,\pi)$ has no isolated vertices for $\pi \geq \pi_+$. This recovers the critical window for the $F$-graph.
Now, let $\delta(F)>0$ and $\varepsilon(F)>0$ as provided by Theorem~\ref{theorem:riordancoupling} and Remark~\ref{rem:constants}, and
\begin{align}
 p_{\pm} = (\pi_{\pm}/(1-n^{-\delta}))^{1/s}.
\end{align}
Note that $p_\pm$ defines a critical window for the disappearance of the last vertex not contained in a copy of $F$, which is easily obtained from Theorem 18 and the proof of Theorem 10 in \cite{riordan2022random}.

\subsection{Bad events}\label{subsec:bad}
We define five bad events which cause our coupling to fail.
The \emph{nullity} of $H$ is
\[
 n(H) = (r-1)e(H)+c(H) - v(H).
\]

\begin{definition}[{Avoidable configuration, \cite{riordan2022random}}]\label{def:avoida}
 Let $H$ be a (labeled multi-)hypergraph. Then $S\subseteq H$ is an \emph{avoidable configuration} if $S$ is connected, $n(S)>1$ and $e(S)\le 2^{r+1}$.
\end{definition}

For example, any two distinct hyperedges on the same vertex set form an avoidable configuration of nullity $r-1 \geq 2$.
The following lemma states that  sufficiently sparse random $r$-uniform hypergraphs do not contain avoidable configurations.

\begin{lemma}[{\cite[Lemma 12]{riordan2022random}}] \label{lem:noavoida}
 For each fixed $r\geq 2$, there is an $\varepsilon>0$ with the following property. If $H=H_r(n,\pi)$ with $\pi=\pi(n)\leq n^{-(r-1)+\varepsilon}$, then whp $H$ contains no avoidable configurations.
\end{lemma}
While avoidable configurations already cause Riordan's static coupling to fail, the following pairs are an additional source of worry in the process coupling.

\begin{definition}[Partner $F$-edges]
 Two $F$-edges in an $F$-graph are \emph{partner $F$-edges} if they share exactly $u$ vertices.
\end{definition}

\begin{remark}\label{remark:avoidableconfigs}
 Note that $F$-edges cannot overlap too much without causing avoidable configurations. So, unless a bad event occurs, the following holds.
 No two $F$-edges in $H_F(n,\pi)$ share more than two vertices.
 Thus, there are no partner $F$-edges for $u>2$. Further, for $u=2$, each $F$-edge has at most one partner $F$-edge, and all pairs of partner $F$-edges are vertex-disjoint. 
\end{remark}

Having introduced all relevant notions, we can finally define the \emph{bad event} $\mathcal B$ as 
\[
 \mathcal{B}= \Bone(\pi_+) \cup \Btwo \cup \Bthree \cup \Bfour \cup \Bfive,
\]
where the individual events on the right-hand side are defined as follows.
\begin{itemize}[leftmargin=7em,labelsep=1.5em]
\item[$\Bone(\pi)$:]  There exists a vertex of degree more than 
$M\pi +\max\left(M\pi,3\log n\right)$.
\item[$\Btwo$:] There exists an {avoidable configuration}.
\item[$\Bthree$:] There are more than $(\log n)^{8 g(n) }$ low-degree vertices, where a low-degree vertex is a vertex of degree  at most $7g(n)$.
\item[$\Bfour$:] There are more than $\log^3 n$ pairs of partner $F$-edges.
\item[$\Bfive$:] There exists an isolated vertex.
\end{itemize}
The next result ensures that the $F$-graphs at the start, the end and well beyond the critical window have the properties required for the proof.

\begin{lemma} \label{lemma:Hnotbad}
 Let $H_F=H_F(n, \pi)$ with $\pi \le n^{1-r+o(1)}$, then we have $H_F \notin \Bone(\pi)\cup\Btwo$ whp.
 Further, we have $H_F \notin\mathcal{B}$ whp for $\pi=\pi_+$.
\end{lemma}
\begin{proof}
 By Lemma~\ref{lem:noavoida}, we have $\tilde H_F(n,\pi)=\hat H_F(n,\pi)$ whp as $H_F$ does not contain multiple $F$-edges on the same vertex set. Since $\hat H$ is distributed as $H_r(n,\pi')$  by Lemma~\ref{lem:pi'}, the assertion follows from \cite[Lemma 2.7]{heckel2024hitting}.
\end{proof}

\section{Static coupling}\label{Sec:Stat}

In this section, we review Riordan's coupling as well as the properties that are relevant for our setting. In Subection~\ref{Sec:Riordan}, we recap Riordan's coupling, as it will form the basis of the process coupling in Section \ref{Sec:Proc}. In Lemma \ref{lemma:lowdegnoextra} in Subsection \ref{Sec:Lemmalow}, we establish that the copies of $F$ in $G=H_u(n,p)$ which are not reflected in $H=H_F(n,\pi)$ whp do not contain vertices that are of low degree in $H$.  
It is in this place where we use a novel combinatorial argument to facilitate the proof.

\subsection{Riordan's coupling} \label{Sec:Riordan}
Fix an arbitrary order $h_1, \ldots, h_M$ of all $F$-edges. We couple $G$ and $H$ by going through the $F$-edges in this order: In step $j$ we reveal whether or not $h_j \in H$, as well as \textit{some} information about $A_j$, the event that the corresponding copy of $F$ is in $G$.

\begin{calgorithm}
 For each $j$ from $1$ to $M$:
 \begin{itemize}
    \item Calculate $\pi_j$, the conditional probability of $A_j$ given all information revealed so far.
    \item If $\pi_j \geq \pi$, toss a coin which lands heads with probability $\pi/\pi_j$, independently of everything else. If the coin lands heads, then test whether $A_j$ holds. Include the $F$-edge $h_j$ in $H$ if and only if the coin lands heads and $A_j$ holds.
    \item If $\pi_j < \pi$, then toss a coin which lands heads with probability $\pi$, independently of everything else, and declare $h_j$ present in $H$ if and only if the coin lands heads.
    In any case, do not test whether $A_j$ holds. If the coin lands heads for any $j$, we say that the coupling has failed. 
 \end{itemize}
 After going through steps $1, \ldots, M$, we have decided on all $F$-edges of $H$ by generating the answers 'yes'/'no' for each $h_j$; and we have revealed information on the events $A_j$ by generating answers 'yes'/'no'/'$*$', where $*$ means that we did not decide $A_j$. Now choose $G$ conditional on the tested events $A_i$, that is, where the answer was not $*$. 
\end{calgorithm}

By \cite{riordan2022random}, the algorithm generates the correct distributions of G and H. There, also the following result was established.
\begin{lemma} \label{lemma:cpl_bad}
 For $p, \pi$ as in Theorem \ref{theorem:riordancoupling}, we have $H\in \mathcal B_1(\pi)\cup \mathcal B_2$ if the coupling fails.
\end{lemma}
\begin{proof}
 The assertion is explicitly established in the proof of Theorem 18 in \cite{riordan2022random}.
\end{proof}
 As Riordan’s proof of Theorem~\ref{theorem:riordancoupling} shows, for nice $H\notin \mathcal B_1(\pi)$ we have $\pi \leq \pi_j $ for all conditional probabilities $\pi_j$ except when $\pi_j$ becomes $0$. In other words:
\begin{lemma}\label{lemma:pi0}
 For $p, \pi$ as in Theorem \ref{theorem:riordancoupling} we have $\pi_j=0$ if $\pi_j < \pi$, unless $H \notin \mathcal B_1(\pi)$. 
\end{lemma}

\subsection{Extra copies near low-degree vertices} \label{Sec:Lemmalow}

The aim of this subsection is to provide an analysis of the ``extra'' copies of $F$ in the $u$-graph $G$, that is, the copies which are not represented by $F$-edges in $H$. This leads to the following lemma, which will later imply (in Section~\ref{Sec:ProcFH})  that whp the extra copies in $G$ have no influence on realizing the hitting time.

\begin{lemma}\label{lemma:lowdegnoextra}
 Let $G=H_u(n,p_+)$ and $H=H_F(n,\pi_+)$ be coupled via Riordan's coupling.
 A copy of $F$ in $G$ is an \emph{extra copy} if the corresponding $F$-edge in $H$ is not present.
 Whp, no vertex which is low-degree in $H$ is incident with an extra copy of $F$ in $G$. 
\end{lemma}

Here, a low-degree vertex is understood in the sense of $\mathcal{B}_3$, defined at the end of Subsection~\ref{subsec:bad}

For the most part, the proof of Lemma~\ref{lemma:lowdegnoextra} uses the same strategy as the proof of the corresponding Lemma~3.1 in \cite{heckel2024hitting} for $r$-cliques with $r\geq 4$.
However, parts of the proof in \cite{heckel2024hitting} do not extend to the nice $u$-graphs considered in this contribution.
Here, we introduce a novel argument that even extends to general strictly $1$-balanced $u$-graphs. Details can be found in the proof of Lemma \ref{lemma:different_estimate} below, specifically around Equation~\eqref{eq:edge_count}.

The first ingredient for the proof of Lemma~\ref{lemma:lowdegnoextra} is Lemma~\ref{lemma:extrabound} below, where we bound the probability that the copy of $F$ corresponding to a given $h_j$ becomes an extra copy of $F$ in $G$. The advantage of the upper bound is that it comes in terms of $H$
only, rather than involving considerations of the coupling history. For the special case $F=K_r$ with $r\geq 4$, Lemma~\ref{lemma:extrabound} is \cite[Lemma~3.5]{heckel2024hitting}.
Subject to Lemmas \ref{lemma:cpl_bad} and \ref{lemma:pi0}, the proof is identical to the one given 
in \cite{heckel2024hitting} and will therefore be omitted. Recall $s=e(F)$ and $\delta=\delta(F)$ from Theorem \ref{theorem:riordancoupling}.

\begin{lemma}\label{lemma:extrabound} 
 Let $p=p_+$, $\pi=\pi_+$ and $H_0 \notin\Bone(\pi) \cup\Btwo$ be an $F$-graph. For any $j \in [M]$ such that $h_j \notin H_0$ set
 \begin{equation}\label{eq:pistardef}
  \pi_j^* = P\big(A_j \mid \bigcap_{i: h_i \in H_0} A_i\big).
 \end{equation}
 Let $G$ and $H$ be distributed via Riordan's coupling.
 Then we can bound the conditional probability that the copy of $F$ corresponding to $h_j$ is an extra copy in $G$ as
 \[
  P\big(A_j \mid H=H_0 \big) \le \frac{\pi_j^*-\pi}{1-\pi}.
 \]
\end{lemma}
Lemma~\ref{lemma:extrabound} provides sufficient control over the extra copies to  establish Lemma \ref{lemma:lowdegnoextra}.

\begin{proof}[Proof of Lemma~\ref{lemma:lowdegnoextra}]
 Let $Z$ denote the number of vertices that are low-degree in $H$ and at the same time incident to an extra copy of $F$ in $G$. As $\mathcal B^c$ is a whp event and
 \[
  P(Z \geq 1) \leq P(Z \geq 1, H \notin \mathcal B) + P(H \in \mathcal B) \leq \bE[Z \cdot \mathds{1}\{H \notin \mathcal B\}] +P(H \in \mathcal B), 
 \]
 it is sufficient to show that $\bE[Z \cdot \mathds{1}\{H \notin \mathcal B\}]$ tends to zero.

 Fix $H_0\notin \mc B$ and let $L$ be the low-degree vertices in $H_0$. Conditional on $H=H_0$, the low-degree vertices of $H$ are $L$, but we do not know which copies of $F$ will become extra copies. Writing $f(n)\sim g(n)$ for $f(n)=(1+o(1))g(n)$, we use Lemma~\ref{lemma:extrabound} to obtain 
 \begin{align}
  \bE[Z \mid H=H_0 ] 
  \leq \sum_{v\in L }\sum_{\substack{j: h_j \notin H_0, \\v\in h_j }} \frac{\pi^*_j-\pi_+}{1-\pi_+} 
  \sim\sum_{v\in L }\sum_{\substack{j: h_j \notin H_0, \\v\in h_j }} (\pi^*_j - \pi_+)\nonumber\\
  \sim\sum_{v\in L}\sum_{\substack{j: h_j \notin H_0, \\v\in h_j }} (\pi^*_j-p_+^{s}) +\sum_{v\in L}\sum_{\substack{j: h_j \notin H_0, \\v\in h_j }} (p_+^{s} - \pi_+).\label{eq:exclatv}
 \end{align}
 Observe that the summands $p_+^{s} - \pi_+$  in the second sum depend neither on $H_0$ nor on $h_j$. Indeed, we have $p_+^{s} - \pi_+ = \pi_+ \frac{n^{-\delta}}{1-n^{-\delta}}$. As for any $v$, there are at most $\frac{r}{\aut(F)}n^{r-1}$ $F$-edges $h_j$ that contain $v$, 
 the inner sum is bounded by $O(n^{-\delta}\log n)$. Finally, as $H_0 \notin \mathcal B$, it contains at most $(\log n)^{8 g(n)} = n^{o(1)}$ low-degree vertices, so that overall, the second sum on the right-hand side of \eqref{eq:exclatv} is $o(1)$. 

 Thus, it only remains to bound the first sum on the right-hand side of \eqref{eq:exclatv}. This is done in Lemma \ref{lemma:different_estimate} below, where we show that the inner sum is again bounded by some $n^{-c}$, thus decaying faster than the $n^{o(1)}$ growth of the number of low-degree vertices.
 Altogether, this then proves Lemma~\ref{lemma:lowdegnoextra}.
\end{proof}

\begin{lemma}\label{lemma:different_estimate}
 Suppose $H_0\notin \mc B_1(\pi_+) \cup \mc B_2$. Then for any low-degree vertex $v$ of $H_0$,
 \begin{equation}\label{eq:sumforexclatv}
  \sum_{\substack{j: h_j \notin H_0, \\v\in h_j }}(\pi^*_j-p_+^{s}) \leq n^{-\frac{1}{s}+o(1)}.
 \end{equation}
\end{lemma}
\begin{proof}
 Let $G_0=G(H_0)$ be the $u$-graph from Definition \ref{def:f-to-hyper}. The conditioning in the definition of $\pi^*_j$ in \eqref{eq:pistardef} now means conditioning on all edges of $G_0$ being present. 

 Let $v$ be a low-degree vertex and $h_j \notin H_0$ such that $v \in h_j$.
 Further, let $F_j$ be the copy of $F$ corresponding to the $F$-edge $h_j$, and let $E_j$ be the edges of $F_j$. Observe that
 \begin{align*}
  \pi^*_j - p_+^{s} 
  = \bP\Big(A_j\Big|\bigcap_{i:h_i\in H_0} A_i\Big) - p_+^{s}
  = p_+^{|E_j\setminus E(G_0)|} - p_+^{s}
  \leq p_+^{|E_j\setminus E(G_0)|}.
 \end{align*}
 In the following, we may assume that $|E_j\setminus E(G_0)|>0$, otherwise $h_j\notin H_0$ exists as a copy of $F$ in $G_0$, which would only be possible if $H_0$ contains an avoidable configuration by \cite[Lemma~19]{riordan2022random} (since $F$ is nice).
 The second case that we may exclude, is the other extreme where $|E_j\setminus E(G_0)|=s$, as the actual contribution of such terms to \eqref{eq:exclatv} is 0. 
    
 For all other cases, consider the $u$-graph $S=(h_j,E_j\cap E(G_0))$. It has $z$ connected components, say, where component $i \in [z]$ has $r_i$ vertices and $s_i$ $u$-edges, and $r_1+\dots+r_z=r$. Also, any connected component of $S$ is a subgraph of $F_j$, hence $s_i\leq d_1(F)(r_i-1)$, or equivalently $(r-1)s_i\leq (r_i-1)s$. 
 Moreover, we have excluded the two cases where $r_1=\dots=r_z=1$, which corresponds to  $|E_j\setminus E(G_0)|=s$, and $S=F_j$, which corresponds to $|E_j \setminus E(G_0)| = 0$.
 Thus, $(r-1)s_i < (r_i-1)s$ for at least one $i$,
 and we obtain (for all relevant cases)
 \begin{align}\label{eq:edge_count}
  (r-1)|E_j\setminus E(G_0)| 
  &= (r-1)\left(s - \sum_{i=1}^z s_i\right) \geq (r-1)\left(s - d_1(F) \sum_{i=1}^z (r_i-1)\right) + 1 \nonumber \\
  &= (r-1)\left(s - d_1(F)(r-z)\right) + 1.
 \end{align}
 Dividing both sides of \eqref{eq:edge_count} by $(r-1)$ yields
 \[
  |E_j\setminus E(G_0)|\geq s - d_1(F)(r-z) + \frac{1}{r-1}.
 \]
 Observe that the last bound only depends on the number $z$ of components of $S$. For each $z \in [r-1]$, we finally upper bound the number of potential $h_j$ that contain $v$ and give rise to $z$ components within $G_0$. For this, observe that the maximum degree in $G_0$ is $\gD(G_0)\leq(r-1)\gD(H_0) = n^{o(1)}$, where we use $\gD$ to denote the maximum degree.
 Next, note that the number of possible vertex sets for the component containing $v$ is bounded by $(r(u-1)\gD(G_0))^{r-1}=n^{o(1)}$. 
 Indeed, the $(k+1)$-th vertex needs to be chosen among the at most $ k(u-1)\gD(G_0)$ neighbours to the previously chosen $k$ vertices. For each of the remaining $z-1$ components, we have an additional choice of at most $n$ vertices for the initial vertex and the bound from above for the anchored component. Altogether, there are thus at most $n^{z-1+o(1)}$ possibilities for the choice of $h_j$ such that $v\in h_j$ and the corresponding $S$ has $z$ components.
 We conclude that the total contribution to the sum in \eqref{eq:sumforexclatv} coming from all $j$ that yield $z$ components is at most 
 \begin{align*}
  n^{z-1+o(1)} p_+^{s-d_1(F)(r-z)+\frac{1}{r-1}}
  &\leq n^{z-1+o(1) + \left(-\frac{1}{d_1(F)}+o(1)\right)\left(s-d_1(F)(r-z)+\frac{1}{r-1}\right)}
  \leq n^{-\frac{1}{s}+o(1)}.
 \end{align*}
 Finally, there are less than $r$ different values for $z$; thus we obtain \eqref{eq:sumforexclatv}.
\end{proof}

\section{Process coupling} \label{Sec:Proc}

We now turn to the dynamic part, namely the coupling of the random $u$-graph process $(G_t)_{t=0}^N$ and the random $F$-graph process $(H_t)_{t=0}^M$. While the first part of the ``3-step-approach'' from \cite{heckel2024hitting} requires (minor) $F$-specific adaptations that are based on Section \ref{Sec:Stat},
parts two and three (Section \ref{Sec:Proc2}) can be used in a ``black box''-manner, harnessing a coupling of the simple, unlabeled random hypergraph process and the random $F$-graph process $(H_t)_{t=0}^M$ as described in Section \ref{Sec:ProcFH}.

\subsection{Coupling the \textit{F}-graph process and the random hypergraph process} \label{Sec:ProcFH}
The following proposition establishes a coupling between the initial parts of the random $F$-graph process and the random hypergraph process.
This simplifies the proof of Proposition \ref{prop:coupling1} below, but much more than that, this allows to directly access the results in \cite{heckel2024hitting} for the remainder, rather than reiterating the proofs.
Let
\[
 T_{E} = \mathrm{min}\{ t : E_t \text{ has no isolated vertices}\}
\]
be the hitting time of having no isolated vertices left in a simple, unlabeled random hypergraph process $(E_t)_{t\geq 0}$.

\begin{proposition}\label{prop:F-H_coupling}
 We may couple the random $F$-graph process $(H_{t})_{t=0}^{M}$ and the random hypergraph process $(E_t)_{t=0}^{\binom{n}{r}}$ such that whp both of the following two properties hold:
 \begin{itemize}
  \item[(a)] The hitting times at which the last isolated vertex disappears agree: $T_{H} = T_E$.
  \item[(b)] The unlabeled processes agree until time $T_H+\lfloor g(n)n\rfloor$, i.e.
  \[
   (\tilde H_{t})_{t=0}^{T_{H} + \lfloor g(n) n\rfloor} = (E_t)_{t=0}^{T_{H} + \lfloor g(n)n\rfloor}.
  \]
 \end{itemize}
\end{proposition}

\begin{proof}
 Let $E=\hat H_F$ be the simple hypergraph that is obtained from $H_F=H_F(n, 2\pi_+)$ by label and multi-hyperedge removal (see Definition \ref{def:f-to-hyper}). Recall from Lemma \ref{lem:pi'} that $E$ is distributed like $H_r(n,\pi')$ with $\pi'=1-(1-2\pi_+)^{r!/\aut(F)}.$ 
 By Lemma \ref{lem:noavoida}, whp $H_F$ contains no avoidable configurations and thus in particular no multi-hyperedges. So, whp, $\tilde H_F = E$ and $v(H_F) = v(E)$.

 We then order the hyperedges of both hypergraphs uniformly to obtain the initial parts $(H_{t})_{t=0}^{v(H_F)}$ and $(E_t)_{t=0}^{v(E)}$ of the processes, and add $F$-edges and hyperedges independently for both processes thereafter to obtain the full processes. It remains to show that, whp, $T_H=T_E$.
 Note that $T_F \leq M\pi_+$ whp and that $v(H_F)$ is distributed like $\Bin(M,2\pi_+)$, so we have $v(H_F) \geq T_H + \lfloor g(n) n\rfloor$ whp.
 But this yields the result since the processes match in this part and the property of a vertex being isolated does not depend on the hyperedge label, meaning the choice of the label given its vertex set.
\end{proof}

Note that we can directly apply the results in \cite{heckel2024hitting} to $(E_t)_t$, and Proposition \ref{prop:F-H_coupling} ensures that these directly carry over to the relevant part of $(\tilde H_t)_t$. Then, as just discussed in the proof, we observe that the choice of the labels does not affect the relevant properties, say, a vertex being isolated or a graph containing a perfect matching.

\subsection{Process coupling -- step 1} \label{Sec:Proc1}
We turn to the adjustments required to establish the coupling of the random $u$-graph process $(G_t)_{t=0}^N$ and the random $F$-graph process $(H_t)_{t=0}^M$.
The following result combines Propositions 4.1, 4.2 and 8.6 in \cite{heckel2024hitting}, and in particular reflects the observation that whp no partner $F$-edges exist for $u>2$ (see introduction to Section 8 in \cite{riordan2022random}).

\begin{proposition}\label{prop:coupling1}
 We may couple the random $u$-graph process $(G_t)_{t=0}^N$ and the random $F$-graph process $(H_t)_{t=0}^M$ so that the following holds.
 Let $\mathcal E=H_\HitH\setminus\cl(G_\HitG)$ and let $\mathcal F$ be the set of $F$-edges in $H_{T_H}$ that have a partner $F$-edge in $H_{\HitH+\left \lfloor g(n)n \right \rfloor}\setminus H_\HitH$. Then we have $\mc E\subset\mc F$ whp.
\end{proposition}

\begin{proof}
 We follow \cite[Section 4]{heckel2024hitting} and focus on the relevant modifications, rather than reiterating the technical details of the proof.
 First, we apply Riordan's coupling to $H_F=H(n, \pi_+)$ and $G=G(n,p_+)$.
 Subsequently, we uniformly order the $F$-edges of $H_F$ and the $u$-edges of $G$ to obtain the initial parts of the processes. This is sufficient as whp $H_F$ contains no isolated vertices and in $G$, every vertex is contained in a copy of $F$. Moreover, we have $|H_F| - T_H  < g(n)n$ whp.

 As in \cite[Section 4.1.1]{heckel2024hitting}, we construct a $u$-edge order after the introduction of additional ``dummy'' $u$-edges for $F$-edges in $H_F$ that overlap in at least $u$ vertices:  For $u=2$, the dummy edges are the shared edges of partner $F$-edges in $H_F$. For $u>2$, whp $H_F$ contains no $F$-edges that overlap in at least $u$ vertices (see Remark \ref{remark:avoidableconfigs} and Section 8.3 in \cite{heckel2024hitting}),
 so there is no need to consider dummy $u$-edges in the first place. Once the set of dummy $u$-edges is determined, the order again arises from the introduction of auxiliary times in $[0,1]$, that not only couple the orders, but also put the $F$-graph process and the $u$-graph process on the same time scale.
 The existence of these auxiliary times and the induced order rests solely on the result that $H \notin \mathcal B$ whp -- which implies that Riordan's coupling succeeds.

 The main remaining step is to show that in auxiliary time, the time $t_H \in [0,1]$ at which the last isolated vertex in the $F$-graph process disappears coincides with the time $t_G \in [0,1]$ at which the last \emph{$F$-isolated} vertex, that is, the last vertex not contained in a copy of $F$, disappears in the random $u$-graph process (\cite[Section 4.1.3]{heckel2024hitting}). 
 For this part, we choose an auxiliary time $t_{-}$ such that the $F$-graph process at time $t_-$, say $H_-$, is essentially the $F$-graph at the beginning of the ``critical window''. In particular $H_-$, has isolated vertices whp. To show that $t_G=t_H$ whp, it is sufficient to establish that whp the following two properties hold:
 \begin{itemize}
  \item[(a)] At auxiliary time $t_-$, the set of isolated vertices in the $F$-graph process agrees with the set of $F$-isolated vertices in the $u$-graph process.
  \item[(b)] All isolated vertices of $H_-$ get their first $F$-edge in the $F$-graph process at the same auxiliary time as they get their first copy of $F$ in the $u$-graph process.
 \end{itemize}
 In essence, the main steps for these observations rely on certain hypergraph properties from \cite{heckel2024hitting} and their $F$-graph versions as provided by Proposition \ref{prop:F-H_coupling}. More specifically, let $G_-$ denote the random $u$-graph process at auxiliary time $t_-$. To argue that no isolated vertex in $H_-$ is $F$-isolated in $G_-$, we observe that whp any isolated vertex in $H_-$ is a low-degree vertex in $H_F$ (\cite[Lemma 4.5]{heckel2024hitting} together with Proposition \ref{prop:F-H_coupling}). 
 According to Lemma \ref{lemma:lowdegnoextra}, no such low-degree vertex is incident to an extra copy of $F$ in $G$.
 Moreover, by \cite[Lemma 2.9]{heckel2024hitting} and Proposition \ref{prop:F-H_coupling}, whp none of these $F$-edges is a partner $F$-edge.
 This shows that whp, for $t\ge t_-$, the isolated vertices in $H_t$ are $F$-isolated in $G_t$.
 To argue that $G_-$ does not contain any additional $F$-isolated vertices, we observe that whp all non-isolated vertices in $H_-$ are incident with at least one non-partner $F$-edge (\cite[Lemma 4.6]{heckel2024hitting} together with Proposition \ref{prop:F-H_coupling}). As this non-partner $F$-edge appears at the same auxiliary time in the $F$-graph process as its copy of $F$ in the $u$-graph process, this implies that any non-isolated vertex of $H_-$ is covered by a copy of $F$ in $G_-$. 

 Due to the whp absence of partner $F$-edges for $u>2$, the discussion above yields $\mc E=\mc F=\emptyset$.
 For $u=2$, the pairs of partner $F$-edges emerge ``earlier'' in the $(H_t)_t$ than their counterparts in $(G_t)_t$ in the auxiliary time construction. 
 Thus, using $t_H=t_G$ whp from above yields $\mc E\subset\mc F$.
 The details for the coupling property can be found in \cite[Section 4.1.2]{heckel2024hitting}, and
 \cite[Section 4.1.4]{heckel2024hitting} explains how $\mc E\subset\mc F$ is derived.
\end{proof}

\subsection{Process coupling - steps 2 \& 3}\label{Sec:Proc2}
Now, we use Proposition \ref{prop:F-H_coupling} to obtain a very short proof for the remaining process coupling steps in \cite{heckel2024hitting}.
\begin{proposition}\label{prop:couplingrandomset}
 There exists a coupling of $(H_t)_t$ and $\mc F$ from Proposition~\ref{prop:coupling1} to another instance $H'_{T_{H'}}$ of the stopped random $F$-graph process so that, whp, 
 \[
  H_\HitH \setminus \mc{F} \supseteq H'_{T_{H'}}.
 \]
\end{proposition}

\begin{proof} 
 Note that \cite[Propositions 5.1 and 6.1]{heckel2024hitting} yield Proposition \ref{prop:couplingrandomset} for the simple random hypergraph process.
 But then, the coupling between the initial parts of the random $F$-graph process and the simple random hypergraph process from Proposition \ref{prop:F-H_coupling} completes the proof.
\end{proof}

\section{Proof of Theorem \ref{thm:Ffactor_HT}}

Theorem \ref{thm:Ffactor_HT} follows by combining the couplings in Propositions \ref{prop:coupling1} and \ref{prop:couplingrandomset} exactly as described in \cite[Section 7]{heckel2024hitting}. To summarize, we have obtained a chain of couplings that whp ``embeds'' the stopped $F$-graph process into the stopped $u$-graph process:
\[
 H'_{\HitHP}  \quad \quad \stackrel{\text{whp}}{\subseteq} \quad \quad
 H_{\HitH} \setminus \mc{F} \quad \quad \stackrel{\text{whp}}{\subseteq} \quad \quad
 H_{\HitH} \setminus \mc{E} \quad \quad \subseteq \quad \quad \cl(G_\HitG).
\]

\bibliographystyle{plain}
\bibliography{literature}

\begin{thebibliography}{10}

\bibitem{bollobas1985}
B\'{e}la Bollob\'{a}s and Andrew Thomason.
\newblock Random graphs of small order.
\newblock In {\em Random graphs '83 ({P}ozna\'{n}, 1983)}, volume 118 of {\em
  North-Holland Math. Stud.}, pages 47--97. North-Holland, Amsterdam, 1985.

\bibitem{devlin2017}
Pat Devlin and Jeff Kahn.
\newblock Perfect fractional matchings in {$k$}-out hypergraphs.
\newblock {\em Electron. J. Combin.}, 24(3):Paper No. 3.60, 12, 2017.

\bibitem{erdos1966existence}
P{\'a}l Erd\H{o}s and Alfr\'{e}d R\'{e}nyi.
\newblock On the existence of a factor of degree one of a connected random
  graph.
\newblock {\em Acta Math. Acad. Sci. Hungar.}, 17:359--368, 1966.

\bibitem{heckel2020random}
Annika Heckel.
\newblock Random triangles in random graphs.
\newblock {\em Random Structures \& Algorithms}, 59(4):616--621, 2021.

\bibitem{heckel2024hitting}
Annika Heckel, Marc Kaufmann, Noela Müller, and Matija Pasch.
\newblock The hitting time of clique factors.
\newblock {\em Random Structures \& Algorithms}, 65(2):275--312, 2024.

\bibitem{janson2000random}
Svante Janson, Tomasz {\L}uczak, and Andrzej Ruci\'nski.
\newblock {\em Random graphs}.
\newblock Wiley-Interscience Series in Discrete Mathematics and Optimization.
  Wiley-Interscience, New York, 2000.

\bibitem{johansson2008factors}
Anders Johansson, Jeff Kahn, and Van Vu.
\newblock Factors in random graphs.
\newblock {\em Random Structures \& Algorithms}, 33(1):1--28, 2008.

\bibitem{kahn2022hitting}
Jeff Kahn.
\newblock Hitting times for {S}hamir's problem.
\newblock {\em Trans. Amer. Math. Soc.}, 375(1):627--668, 2022.

\bibitem{kahn2019asymptotics}
Jeff Kahn.
\newblock Asymptotics for {S}hamir's problem.
\newblock {\em Advances in Mathematics}, 422, 2023.

\bibitem{petersen1891}
Julius Petersen.
\newblock Die {T}heorie der regulären graphs.
\newblock {\em Acta Math.}, 15:193--220, 1891.

\bibitem{riordan2022random}
Oliver Riordan.
\newblock Random cliques in random graphs and sharp thresholds for
  {$F$}-factors.
\newblock {\em Random Structures \& Algorithms}, 61(4):619--637, 2022.

\bibitem{tutte1947factorization}
William~T Tutte.
\newblock The factorization of linear graphs.
\newblock {\em Journal of the London Mathematical Society}, 1(2):107--111,
  1947.

\end{thebibliography}

\end{document}